\documentclass[11pt]{amsart}

\usepackage{amsmath}
\usepackage{hyperref}
\usepackage{xy}
\usepackage{amscd}
\newfont{\sheaf}{eusm10 scaled\magstep1}

\newtheorem{definition}{Definition}[section]

\newtheorem{proposition}{Proposition}[section]

\newtheorem{lemma}[proposition]{Lemma}
\newtheorem{theorem}[proposition]{Theorem}

\newtheorem{remark}{Remark}[section]

\DeclareMathOperator{\Hilb}{Hilb}
\DeclareMathOperator{\Div}{Div}
\DeclareMathOperator{\PGL}{PGL}

\DeclareMathOperator{\Sing}{Sing}
\DeclareMathOperator{\Aut}{Aut}

\DeclareMathOperator{\Pic}{Pic}

\DeclareMathOperator{\Sec}{Sec}

\title[NOTES] {Geometry of  genus 8 Nikulin surfaces and rationality of their moduli}
  \author{Alessandro Verra}
\address{Dipartimento di Matematica e Fisica, Universit\'a Roma Tre,
Largo San Leonardo Murialdo 1 00146 Roma, Italy}
\email{sandro.verra@gmail.com}
 \thanks{Research supported by the PRIN Project  'Geometry of Algebraic Varieties' and by INdAM-GNSAGA}

\begin{document} \maketitle 

 \section{\small Introduction}
 In this note we study Nikulin surfaces of genus $8$ and their moduli. As typical at least in low genus, the family of surfaces to be investigated sits  in a fascinating system of relations to other known geometric families.  Our aim is to unveil one of these relations, namely that occurring between the moduli of Nikulin surfaces of genus 8 and the Hilbert scheme of rational sextic curves in the Grassmannian $G(1,4)$. We will work over an algebraically closed field $k$, $char \ k = 0$. \par
  A Nikulin surface of genus $g$ will be a K3 surface $S$ endowed with a pseudoample,  primitive polarization $\mathcal L$ of genus $g$ and a line bundle $\mathcal O_S(M)$ such that $2M \sim N$, where $N$ is the disjoint union of $8$ copies of $\mathbf P^1$. In particular one has $g \geq 2$. \par The corresponding  moduli space of Nikulin surfaces of genus $g$ is known to be an equidimensional quasi projective variety of dimension $11$.   Its irreducible components are essentially characterized by the intersection index of $\mathcal L$ and $\mathcal O_S(2M)$,  \cite{H, GS, SvG}. \par
 \it We will assume throughout all the paper that the intersection index of $\mathcal L$ and $\mathcal O_S(M)$ is zero. \rm \par The latter  condition defines an integral component of the moduli space which exists for every $g$, differently from the other possible conditions. Moreover this is the unique irreducible component if $g \equiv 0 \mod 4$.
 On the other hand the birational structure of it, when $g$ varies, is quite unknown. Recently its unirationality has been proven for $g \leq 7$,  \cite{FV, FV1}, while no result on the Kodaira dimension seems to be known for $g \geq 9$.   The main result proved in this paper is the following
  \\  \\ \noindent  {THEOREM  I} {\it The moduli space of genus $8$ Nikulin surfaces is rational.} \rm \\ \par
As a consequence, it is natural to ask wether the rationality of these moduli spaces could be true for lower genus $g \leq 7$. Let us briefly describe the method of proof and how it is related to a special family of rational normal sextic curves in the Pl\"ucker embedding of the Grassmannian $G(1,4)$. \par
Assume that $S$ is a general Nikulin surface of genus $8$, then $\mathcal H := \mathcal L(-M)$ is a very ample polarization of genus $6$. Moreover we have  $\mathcal H(-M) \cong \mathcal O_S(A)$, where $A$ is a copy of $\mathbf P^1$ and $A$ has degree $6$ with respect to $\mathcal H$. \par The main idea behind the proof  is to consider the Mukai bundle $\mathcal E \to S$ defined by the polarization $\mathcal H$. As is well known this is a stable rank  $2$ vector bundle with $h^0(\mathcal E) = 5$ and $det \ \mathcal E \cong \mathcal H$. Let $G(1,4) \subset \mathbf P^9$ be the Pl\"ucker embedding of the Grassmannian of lines  of $\mathbf PH^0(\mathcal E)^*$. It turns out that $\mathcal H$ defines an embedding $S \subset \mathbf P^6$, factoring through the classifying map $S \to G(1,4)$ defined by $\mathcal E$. Furthermore, it is also true that the restriction map $r: H^0(\mathcal O_{G(1,4)}(1)) \to H^0(\mathcal H)$ is surjective. Let $\mathbf P^6 \subset \mathbf P^9$ be the linear embedding induced by the dual of $r$, then we have
$$
A \subset S \subset G(1,4) \cap \mathbf P^6 \subset \mathbf P^9.
$$
\par
If $S$ is general  $A$ is embedded as a rational normal sextic and the restriction $\mathcal E_A := \mathcal E \otimes \mathcal O_A$ is the sheaf $\mathcal O_{\mathbf P^1}(3) \oplus \mathcal O_{\mathbf P^1}(3)$. We  will consider the ruled surface $ \mathbf P\mathcal E^*_A$ and its tautological embedding $R$ in $\mathbf P^7 := \mathbf PH^0(\mathcal E_A)^*$. Let $\mathbf P^4 := \mathbf PH^0(\mathcal E)^*$ then the restriction $H^0(\mathcal E) \to H^0(\mathcal E_A)$ uniquely defines a  linear projection
$$
p: R \to \mathbf P^4.
$$ 
\par For a Nikulin surface $S$ the map $p$ is not a generic linear projection of $R$.   As we will see the projected surface $p(R)$ is a scroll with a curve of double points and this is actually a rational normal quartic curve $B$. More precisely $p(R)$ is the complete intersection of a quadric through $B$ and $Sec \ B$, the cubic hypersurface of the bisecant lines to $B$. \par
On the other hand this realization of $p(R)$ is used to show that $A$ admits a one dmensional family  of bisecant lines which are contained in $G(1,4)$. Let $T$ be the intersection of the linear span of $A$
with $G(1,4)$: we will show that $T$ is a smooth threefold. Let us say that $A \subset T$ is an \it embedding of special type. \rm  Along the paper we describe the special geometry of it and its relations to Nikulin surfaces. One can summarize all that as follows.  \\  \\ \noindent  {THEOREM  II} {\it    Let $A \subset  T$ be a general embedding of special type and let $\mathcal I_A$ be its ideal sheaf. Then:
  \begin{itemize}
  \item[$\circ$] A general member of $\vert \mathcal I_A(2) \vert$ is a Nikulin surface of genus $8$.
  \item[$\circ$] For a general Nikulin surface  $S$ of genus $8$ there exists a special embedding $A \subset T$ such that
 $$
 S \in  \vert \mathcal I_A(2) \vert.
 $$ 
 \end{itemize}   } \rm \par

This is the starting point for a quick  conclusion of the proof of theorem I, saying that the moduli space of genus 8 Nikulin surfaces is rational. 
      \par We consider this note as a step in the study of moduli of Nikulin surfaces of low genus by similar geometric methods. As is well known a general member of the Hilbert scheme of $S$ in $G(1,4)$ is a general K3 surface of genus $6$. Furthermore, by Mukai linear section theorem, a general K3 surface of genus $7 \leq g \leq 10$ can be constructed as a linear section of a suitable homogeneous space $\mathbb S_g$ and  such a realization is unique up to $Aut \ \mathbb S_g$. \par Describing Mukai realizations of Nikulin surfaces of genus $g = 9, 10, 11$ appears to be, as in the case $g = 8$, very rich of geometric connections and interesting for the study of the corresponding moduli spaces. \par More in general the moduli space of Nikulin surfaces of genus $g$ deserves to be studied.  
In particular it could be tempting to compare the family of the moduli spaces of Nikulin surfaces of genus $g$ and the family of the moduli spaces of Enriques surfaces endowed with a genus $g$ polarization, since the two classes of surfaces admit some analogies. \par For instance, in the case of Enriques surfaces, a recent theorem of Hulek and Gritsenko, \cite{GH}, shows that the irreducible components of the moduli spaces, when $g$ varies,  distribute in finitely many birational classes: what happens in the case of Nikulin surfaces? Hopefully further work will be developed about the previous questions in the future. \par  
 \section{\small Notations and preliminary results} A \it K3 surface of genus $g$ \rm is a pair $(S, \mathcal L)$ such that $S$ is a K3 surface and $\mathcal L := \mathcal O_S(C)$ is a pseudoample, primitive element of $Pic \ S$ such that $C^2 = 2g - 2$. The moduli space of K3 surfaces of genus $g$ is an irreducible, quasi-projective variety of dimension $19$. It will be denoted as usual by $\mathcal F_g$. \par
As already outlined in the introduction a \it Nikulin surface of genus $g$ \rm will be a triple $(S, \mathcal L, \mathcal M)$ such that: \medskip \par
\begin{enumerate} \it
\item $(S, \mathcal L)$ is a K3 surface of genus $g$.
\item $\mathcal M := \mathcal O_S(M)$, where $2M \sim N_1 + \dots + N_8$ and  $N_1 \dots N_8$ are disjoint copies of $\mathbf P^1$.
\item $\mathcal L$ and $\mathcal M$ are orthogonal in $Pic \ S$.
\end{enumerate}
\medskip \par
Note that $N_iN_j = -2\delta_{ij}$ and that $M^2 = -4$. We fix the notation
$$
N := N_1 + \dots + N_8.
$$
\begin{lemma} $h^i(\mathcal O_S(M)) = 0$ for $i = 0, 1, 2$. \end{lemma}
\begin{proof} Since $N$ is isolated and $2M \sim N$, we have $h^0(\mathcal O_S(M)) \leq 1$. Let $h^0(\mathcal O_S(M)) = 1$, then $M \sim E$ for an effective divisor $E$ and
 $2E = N$. Since $N$ is reduced this is impossible,  hence $h^0(\mathcal O_S(M)) = 0$. Moreover we have $h^0(\mathcal O_S(-2M)) = h^0(\mathcal O_S(-N)) = 0$ so that $h^0(\mathcal O_S(-M)) = 0$. This implies $h^2(\mathcal O_S(M)) = 0$. Finally $h^1(\mathcal O_S(M)) = 0$ follows from $\chi(\mathcal O_S(M)) = 0$. \end{proof}
We recall that  the \it Nikulin lattice \rm  is an even lattice of rank $8$ generated by $n_1 \dots n_8$ and $m = \frac 12(n_1 + \dots + n_8)$, the product of which is uniquely defined by the condition $n_in_j = -2\delta_{ij}$, \cite{M} 5.3. Notice also that the set of generators $\lbrace n_1, \dots,  n_8, m \rbrace$ is  unique up to multiplying each element by $-1$. Let $i = 1 \dots 8$, clearly $\mathcal O_S(M)$ and $\mathcal O_S(N_i)$ generate a Nikulin lattice in $Pic \ S$. Moreover this set of generators is uniquely defined by the condition $h^0(\mathcal O_S(N_i)) = 1$. We will denote such a lattice by $\mathbb L_S$ and say that $\mathbb L_S$ is the \it Nikulin lattice \rm of $(S, \mathcal L, \mathcal M)$. The next property is well known
\begin{theorem} For a general Nikulin surface of genus $g$ one has  $$ Pic \ S = \mathbb Z \mathcal L \oplus \mathbb L_S. $$ \end{theorem}
It easily follows from the theorem and the preceeding remarks that the assignement $(S, \mathcal L, \mathcal M) \longrightarrow (S, \mathcal L)$ induces a generically injective rational map from the moduli space of  Nikulin surfaces of genus $g$ to $\mathcal F_g$. We will say that the image of this map is the \it Nikulin locus \rm in $\mathcal F_g$ and denote it by
$$
\mathcal F^N_g.
$$
\par
Our aim is to show that $\mathcal F^N_g$ is rational for $g = 8$. To this purpose let us summarize some geometry of  projective models a Nikulin surface $(S, \mathcal L, \mathcal M)$ of genus $g$, cfr. \cite{GS, SvG}.
To begin we have the commutative diagram:
$$
\begin{CD}
{\tilde S'} @>{\nu'}>> {\tilde S} \\
@V{\pi'}VV @V{\pi}VV \\
{S} @>>{\nu}> {\overline S} \\
\end{CD}
$$ 
where $\pi'$ is the double covering defined by $N$ and $\nu$ is the contraction of $N$. Let $E_i = \pi'^{-1}(N_i)$, $i = 1 \dots 8$, then $E_i$ is an exceptional line on the smooth surface $\tilde S'$.
It turns out that $\nu'$ is the contraction of $E_1 + \dots + E_8$ and that $\tilde S$ is a minimal K3 surface. Moreover $\pi$ is the quotient map of a symplectic involution $\iota: \tilde S \to \tilde S$ and its $8$ fixed points are 
$\nu'(E_1) \dots \nu'(E_8)$. Then it follows that $\pi$  is branched exactly on the even set of nodes
$$
o_1 := \nu(N_1) \ , \ \dots \ , \ o_8 := \nu (N_8)
$$
of $\overline S$, in particular $\Sing \ \overline S = \lbrace o_1 \dots o_8 \rbrace$. On $\tilde S$ we fix the polarization
$$
\tilde {\mathcal L } := \pi^*\overline { \mathcal L},
$$
where $\overline {\mathcal L}$ is the line bundle $\nu_* \mathcal L$. Then $(\tilde S, \tilde {\mathcal L})$ is a K3 surface of genus $2g - 1$ and $\iota^* \tilde {\mathcal L} \cong \tilde {\mathcal L}$. The action of $\iota$
on $H^0(\tilde {\mathcal L})$ can be described as follows:
$$
H^0(\tilde {\mathcal L}) = \pi^* H^0(\overline {\mathcal L}) \oplus H^0(\mathcal I_f \otimes \tilde {\mathcal L}),
$$
where $f = \pi^{-1}(\Sing \overline S)$ is the set of the fixed points of $\iota$ and $\mathcal I_f$ is its ideal sheaf in $\tilde S$. The previous summands respectively are the $+1$ and $-1$ eigenspaces of $\iota$.
We have $\dim H^0(\overline {\mathcal L}) = g + 1$ and $\dim H^0(\mathcal I_f \otimes \tilde {\mathcal L}) = g-1$. \par Now let us consider the rational  map $f_{\tilde {\mathcal L}}: \tilde S \to \mathbf P^{2g-2}Ê:=  \mathbf PH^0(\tilde {\mathcal L})^*$ defined by $\tilde {\mathcal L}$ and the linear subspaces $\mathbf P^g := \mathbf PH^0(\tilde {\mathcal L})^*$ and $\mathbf P^{g-2} := \mathbf PH^0(\mathcal I_f \otimes \tilde {\mathcal L})^*$ of
$\mathbf P^{2g-1}$ defined by the previous direct sum.  Then we can add to the picture the following commutative diagram:
$$
\begin{CD}
{S} @<{\pi' \circ \nu^{-1}}<< {\tilde S} @>{\pi}>> {\overline S} \\
@VV{f_{\mathcal L(-M)}}V    @V{f_{\tilde {\mathcal L}}}VV  @V{f_{\overline {\mathcal L}}}VV \\
{\mathbf P^{g-2}} @<{\pi_-}<< {\mathbf P^{2g-1}} @>{\pi_+}>> {\mathbf P^g} \\
\end{CD}
$$
Here $\pi_+$ and $\pi_-$ respectively are the linear projections of centers $\mathbf P^{g-2}$ and $\mathbf P^g$. Notice that $\iota$ acts on $\mathbf P^{2g-1}$ as a projective involution and that $\mathbf P^{g-2}$ and
$\mathbf P^g$ are its projectivized eigenspaces. It is easy to see that the vertical maps  are defined, from left to right,  by the line bundles $\mathcal L(-M)$, $\tilde {\mathcal L}$, $\overline {\mathcal L}$. We omit any further detail. From now on we fix the notations
$$
\mathcal H := \mathcal L(-M) \ , \ \mathcal A := \mathcal L(-2M).
$$
We will use the following well known facts.
\begin{proposition} For a general Nikulin surface of genus $g \geq 5$ the line bundle $\mathcal H$ is very ample. \end{proposition}
\begin{proof} See \cite{GS} lemma 3.1 \end{proof}
\begin{proposition} For a general Nikulin surface of genus $g \geq 8$ a general member of $\vert \mathcal A \vert$ is a smooth irreducible curve of genus $g - 8$.   \end{proposition}  
\begin{proof} We only sketch the standard proof: assume $g \geq 8$, then it follows from Riemann-Roch that $h^0(\mathcal A) \geq 1$. Let $A \in \vert \mathcal A \vert$, then $A + N$
belongs to $\vert \mathcal L \vert$ and we have the standard exact sequence
$$
0 \to \mathcal A \to \mathcal L \to \mathcal O_N(A) \to 0.
$$
Moreover $\mathcal O_N(A)$ is the trivial sheaf $\mathcal O_N$ and it is easy to show that the restriction $H^0(\mathcal L) \to H^0(\mathcal O_N)$ is surjective. Passing to the associated long
exact sequence it follows $h^1(\mathcal A) = h^1(\mathcal L) = 0$. This implies that $\vert \mathcal A \vert$ is base point free and the statement for $g \geq 9$. For $g = 8$ we have $A^2 = -2$
so that $A$ is an isolated curve with $p_a(A) = 0$. If $A$ is not integral one can deduce that $\Pic S$ has rank $\geq 10$: a contradiction for a general $S$.
\end{proof}
Now, for a general Nikulin surface $S$ of genus $g \geq 5$, we consider the embedding
$$
S \subset \mathbf P^{g-2}
$$
defined by $f_{\mathcal H}$. Let $C \in \vert \mathcal L \vert$: since $(C-M)N_i = 1$ it follows that $N_1 \dots N_8$ are embedded as lines. Let $A \in \vert \mathcal A \vert$,  $A$ is embedded as
a curve of degree $2g - 10$. A general $A$ is integral of genus $g - 8$, hence $h^1(\mathcal O_A(1)) = 0$ for degree reasons. Since $\mathcal H(-A) \cong \mathcal O_S(M)$, we have the exact sequence
$$
0 \to \mathcal O_S(M) \to \mathcal O_S(1) \to \mathcal O_A(1) \to 0.
$$
\par On the other hand we know that $h^i(\mathcal O_S(M)) = 0$, $i = 0, 1, 2$. Therefore, passing to the associated long exact sequence, we obtain
$$ 
0 \to H^0(\mathcal O_S(1)) \to H^0(\mathcal O_A(1)) \to 0.
$$
\par  This shows that
\begin{proposition} For a general Nikulin surface $S$ of genus $g \geq 8$ a general $A \in \vert \mathcal A \vert$ is embedded by $f_{\mathcal H}$ as a smooth irreducible curve spanning $\mathbf P^{g-2}$.
\end{proposition}
For a general Nikulin surface $S$ of genus $g \geq 8$ we also point out that
\begin{proposition} In the projective model defined by $f_{\mathcal H}$ the lines $N_1 \dots N_8$ are bisecant lines to a smooth irreducible $A \in \vert \mathcal A \vert$. \end{proposition}
Finally we can also consider the moduli space $\mathcal D'_g$  of triples $(X, \mathcal L, \mathcal H)$ such that $(X, \mathcal L)$ is a K3 surface of genus $g$ and $\mathcal H \in \Pic S$ is a primitive big and
nef element satisfying the following intersection properties:
$$
(\mathcal H , \mathcal H) = 2g - 6 \ , \ (\mathcal H , \mathcal L) = 2g - 2.
$$
\par As is well known this moduli space is an integral quasi projective variety and the assignement $(X, \mathcal L, \mathcal H) \longrightarrow (X, \mathcal L)$ induces a generically injective morphism
$\mathcal D'_g \to \mathcal F_g$. Its image is an integral divisor we will denote as
$$
\mathcal D_g,
$$
notice that $\Pic X = \mathbb Z \mathcal L \oplus \mathbb Z \mathcal H$ for a general triple $(X, \mathcal L, \mathcal H)$, cfr. \cite{ H, BV}. \par
It is clear that we have the inclusions
$$
\mathcal F^N_g \subset \mathcal D_g \subset \mathcal F_g.
$$
This implies, by semicontinuity and the irreducibility of $\mathcal D'_g$, that the above propositions 2.3, 2.4, 2.5 extend verbatim from the case of a general Nikulin surface to that of  a general triple
$(X, \mathcal L, \mathcal H)$ and to its line bundles $\mathcal H$ and $\mathcal A := \mathcal L(-2M)$, where $\mathcal O_S(M) := \mathcal L \otimes \mathcal H^{-1}$. For such a general triple let
$$
X \subset \mathbf P^{g-2}
$$
be the embedding defined by $\mathcal H$. If $X$ is a Nikulin surface then $X$ contains $8$ disjoint lines. This condition on $(X, \mathcal L, \mathcal H)$ is not enough to have a
Nikulin. For $g \geq 8$ we have $\mathcal A \mathcal A = 2g - 18 \geq -2$, hence $\vert \mathcal A \vert$ is not empty. The next result characterizes the Nikulin locus $\mathcal F^N_g$ in $\mathcal D_g$, cfr. \cite{GS} 3.2. 
    \begin{theorem} Let $g \geq 8$: a general triple $(X, \mathcal L, \mathcal H)$ defines a point in $\mathcal F^N_g$ iff $X$ contains eight disjoint lines $N_1 \dots N_8$ which are
    bisecant to a curve $A \in \vert \mathcal A \vert$, that is, iff $AN_1 = \dots = AN_8 = 2$. \end{theorem} 
\begin{proof}  Assume $AN_1 = \dots = AN_8 = 2$ with $A \in \vert \mathcal A \vert$ and consider $A' := 2H - A - N$, where $H \in \vert \mathcal H \vert$ and $N := N_1 + \dots + N_8$. One computes $A^2 = A'^2 = AA' = 2g - 18$ and $HA = HA' = 2g -10$. So we have $(A - A')^2 =0$ and $H(A - A') = 0$.  This implies $A \sim A'$ since $H$ is pseudoample. Then it follows  $N \sim 2H - 2A$ and $X$ is a Nikulin surface. The converse is immediate. \end{proof}
    \section{\small Nikulin surfaces of genus 8 and rational normal sextics}
Let $S$ be a general Nikulin surface of genus $8$ embedded in $\mathbf P^6$ by $\vert \mathcal H \vert $, then $\vert \mathcal A \vert$ contains a unique element $A$ which is embedded as a rational normal sextic.
$S$ is in the irreducible component of its Hilbert scheme, the general point of which general point is a smooth K3 surface $(X, \mathcal O_X(1))$ of genus $6$. Moreover one has $$ \Pic X \cong \mathbb Z \mathcal O_X(1). $$
\par  The Mukai-Brill-Noether theory for curves and K3 surfaces is definitely performed in genus $6$, cfr. \cite{Mu, Mu1}. It can be summarized as follows. A smooth hyperplane section $H$ of $X$ is a canonical curve of genus $6$:
\medskip \par
{\em Case (1)} Assume that $H$ is not trigonal nor biregular to a plane quintic. Then $H$ is generated by quadrics, moreover there exists exactly one $H$-stable rank $2$ vector bundle $\mathcal E$ on $X$ such that: 
 \begin{enumerate} \it \item[(i)] $det \ \mathcal E \cong \mathcal O_Y(1)$; \item[(ii)] $h^0(\mathcal E) = 5$ and $h^i(\mathcal E) = 0$ for $i \geq 1$; \item[(iii)]  the determinant map
$ det: \wedge^2 H^0(\mathcal E) \to H^0(\mathcal O_{\mathbf P^6}(1))$ is surjective. \end{enumerate} \medskip \par  
Let $G(1,4) \subset \mathbf P^9 := \mathbf P \wedge^2 H^0(\mathcal R)^*$ be the Pl\"ucker embedding of the Grassmannian of $2$-dimensional subspaces of $H^0(\mathcal E)^*$. By (iii) the dual of $det$ induces a linear embedding $\delta: \mathbf P^6 \to \mathbf P^9$, moreover the construction yelds the commutative diagram
$$
\begin{CD}
{\mathbf P^6} @>{\delta}>> {\mathbf P^9} \\
@AAA @AAA \\
X @>{f_{\mathcal E}}>> {G(1,4)} \\
\end{CD}
$$
where the vertical maps are the inclusions and $f_{\mathcal E}$ is the embedding defined by $\mathcal E$. Fixing the identifications $\mathbf P^6 := \delta(\mathbf P^6)$ and $X := f_{\mathcal R}(X)$ let us say in a simpler way that
$$
X \subset \mathbf P^6 \cap G(1,4) \subset \mathbf P^9.
$$
Let
$$
T = \mathbf P^6 \cdot G(1,4) \subset \mathbf P^9,
$$
Mukai theory in genus $6$ says also that: 
\medskip \par \begin{enumerate} \it \item[(iv)] $X$ is a quadratic section of $ T$,  \end{enumerate} \medskip \par
Since $X$ is a smooth quadratic section of $T$ it follows that $T$ is an integral $3$-dimensional linear section of $G(1,4)$ with isolated singularities. Actually $T$ is a smooth Del Pezzo threefold of
degree $5$ if $X$ is sufficiently general.  
\medskip \par
{\em Case (2) } Assume that $H$ is either trigonal or biregular to a plane quintic. Then $H$ has Clifford index $1$ and the following property holds true: \it
\begin{itemize} \it \item[$\circ$] there exists an integral curve $D \subset X$ such that either $DH = 3$ and $D^2 = 0$ or $DH = 5$ and $D^2 = 2$. \rm
\end{itemize} \rm
A general Nikulin surface of genus $8$ occurs in case (1). 
\begin{proposition} Let $S \subset \mathbf P^6$ be a general Nikulin surface of genus $8$ embedded by $f_{\mathcal H}$. Then $S$ is a quadratic section of a threefold $T$ as above.  
 \end{proposition}
\begin{proof} $\Pic S$ is the orthogonal sum of rank 9 $\mathbb Z \mathcal L \oplus \mathbb L_S$, where $\mathbb L_S$ is the Nikulin lattice generated by $\mathcal O_S(M), \mathcal O_S(N_1) \dots \mathcal O_S(N_8)$. A standard computation we omit, shows that no divisor $D$ exists such that $D^2 = 0$ and $DH = 3$ or $D^2 = 2$ and $DH = 5$. This excludes case (2).  \end{proof}
From now on we assume that $S$ is a general Nikulin surface of genus $8$, in particular we will assume that $S$ occurs in case (1) and that $\Pic S \cong \mathbb Z \mathcal L \oplus \mathbb L_S$. We also assume that $S$ is embedded in $\mathbf P^6$ by $\vert \mathcal H \vert$. Then $S$ contains the rational normal sextic $A$ which is the unique element of $\vert \mathcal A \vert$. We want to study the restriction
$$
\mathcal E_A := \mathcal E \otimes \mathcal O_A
$$
of the Mukai bundle $\mathcal E$ and discuss the possible cases. Of course we have $\mathcal E_A = \mathcal O_{\mathbf P^1}(m) \oplus \mathcal O_{\mathbf P^1}(n)$ with $m + n = 6$.
\begin{lemma}ÊOne has $m, n \geq 0$ so that $h^0(\mathcal E_A) = 8$ and $h^1(\mathcal E_A) = 0$. \end{lemma}
\begin{proof} Consider the commutative diagram
$$
\begin{CD}
{\wedge^2 H^0(\mathcal E)} @>{\wedge^2r}>>{\wedge^2 H^0(\mathcal E_A)} \\
@V{det}VV @V{det_A}VV \\
{H^0(det \ \mathcal E)} @>r>> {H^0(det \ \mathcal E_A)} \\
\end{CD}
$$
The restriction map $r$ is an isomorphism by 2.5 and $det$ is surjective. This implies $m, n \geq 0$: otherwise $det_A$ would be the zero map.
\end{proof}
Now we consider the surface $\mathbb P_A := \mathbf P \mathcal E^*_A$ and its tautological map 
$$ u_A: \mathbb P_A \to \mathbf P^7 := \mathbf PH^0(\mathcal E_A)^*. $$ 
  
Since $m$ and $n$ are non negative, $u_A$ is a generically injective morphism with image a rational normal scroll of degree $6$. For it we fix the notation
$$
R := u_A(\mathbb P_A).
$$
We can assume $m \leq n$. Notice that $m$ is the minimal degree of a section of $\mathbb P_A$. If $m = 0$ then $R$ is a cone over a rational normal sextic. If $m \geq 1$ then $u_A$ is an embedding. We consider the standard exact sequence
$$
0 \to \mathcal E(-A) \to \mathcal E \to \mathcal E_A \to 0.
$$
\begin{lemma} The associated long exact sequence is the following:
$$
0 \to H^0(\mathcal E) \to H^0(\mathcal E_A) \stackrel {\delta_A} \to H^1(\mathcal E(-A)) \to 0.
$$
In particular one has $h^0(\mathcal E) = 5$, $h^0(\mathcal E_A) = 8$ and $h^1(\mathcal E(-A)) = 3$.
\end{lemma}
\begin{proof}
Since $\mathcal E(-A)$ is $H$-stable and $H(H-2A) < 0$, it follows $h^0(\mathcal E(-A)) = 0$. Furthermore we know that $h^i(\mathcal E) = 0$ for $i \geq 1$ and
we have $h^1(\mathcal E_A) = 0$ because $m,n \geq 0$. This implies the statement. \end{proof}
Then the coboundary map $\partial_A: H^0(\mathcal E_A) \to H^1(\mathcal E(-A))$ defines a plane
$$ P_A := \mathbf P Im \ \partial_A^* \subset \mathbf P^7.$$
\par Let $\mathbf P^4 := \mathbf PH^0(\mathcal E)^*$. Then, dualizing the sequence and projectivizing the maps,  we  obtain the linear projection
$$
\alpha_A: \mathbf P^7 \to \mathbf P^4 := \mathbf PH^0(\mathcal E)^*,
$$
of center $P_A$. Let $\mathbb P_S := \mathbf P \mathcal E^*$, in turn $\alpha_A$ defines the commutative diagram
$$
\begin{CD}
{\mathbf P^7} @>{\alpha_A}>> {\mathbf P^4} \\
@A{u_A}AA @A{u_S}AA \\
{\mathbb F_A} @>i_{\mathbb F}>> {\mathbb F_S} \\
\end{CD}
$$
where   $i_{\mathbb F}$ is the inclusion $\mathbb F_A \subset \mathbb F_S$ and the vertical arrows are the tautological maps. Furthermore let $G(2,8)$ be the Pl\"ucker embedding of the Grassmannian of $2$-dimensional subspaces in $H^0(\mathcal E_A)^*$ and let $l \subset \mathbf P^7$ be a  general line. Then the assignement $l \longrightarrow \alpha_A(l)$ defines a natural linear projection
$$
\lambda_A: G(2,8) \to G(1,4).
$$
It follows immediately from the previous diagram that
\begin{proposition} The next diagram is commutative:
$$
\begin{CD}
{G(2,8)} @>{\lambda_A}>> {G(1,4)} \\
@A{f_{\mathcal E_A}}AA @A{f_{\mathcal E}}AA \\
{A} @>i>> {S} \\
\end{CD}
$$
\end{proposition}
Here  $i$ is the inclusion map and $f_{\mathcal E_A}$ and $f_{\mathcal E}$ are the maps associated to $\mathcal E_A$ and $\mathcal E$.  We will profit of this construction in the next section,
where the very special feature of the projection $\alpha_A$ will be described.  For the moment we use the previous remarks to describe $\mathcal E_A$  for a general $S$.
\begin{theorem} For a general Nikulin surface of genus $8$ one has $$ \mathcal E_A  = \mathcal O_{\mathbf P^1}(3) \oplus \mathcal O_{\mathbf P^1}(3). $$
\end{theorem}
\begin{proof} We have $\mathcal E_A = \mathcal O_{\mathbf P^1}(m) \oplus \mathcal O_{\mathbf P^1}(n)$ with $0 \leq m \leq n \leq 6$ and $m + n = 6$. It suffices to show that
$R$ is not a cone and that no rational section of degree $1$ or $2$ is contained in it. This indeed implies $m = 3$. To this purpose consider the projected scroll
$R' = \alpha_A(R)$. Since $A$ is embedded in $G(1,4)$ as an integral sextic curve, the degree of $R'$ is six. For any integral variety $Y \subset \mathbf P^4$ we denote by
$\sigma_Y$ the variety in $G(1,4)$ parametrizing the lines intersecting $Y$. Let us exclude the cases $0 \leq m \leq 2$.
\par
$m = 0$. Then the scroll $R'$ is a cone of vertex $o$ and $A$ is contained in $\sigma_o$. But $\sigma_o$ is a linear space of dimension four and $A$ would be a degenerate curve in it, which is excluded.
 \par
$m = 1$. In this case $R'$ contains a line $L$ intersecting every line of its ruling. Consider $\sigma_L$: it is well known that $\sigma_L$ is a cone of vertex a point $l$ over the
Segre embedding $\mathbf P^1 \times \mathbf P^2 \subset  \mathbf P^5$. Since $A \subset \sigma_L$ it follows that $\sigma_L \subset \mathbf P^6 = < A >$. Moreover $\mathbf P^6$  
is the linear space tangent to $G(1,4)$ at the parameter point of $L$. But then $T = \sigma_L$: a contradiction. 
\par
$m = 2$. We can assume that $R'$ contains a smooth conic $K$ intersecting all the lines of the ruling of $R'$. Let $P$ be the supporting plane of $K$, then $S$ is contained in
the codimension $1$ Schubert cycle $\sigma_P$. This is endowed with a ruling of $4$-dimensional smooth quadrics having the dual plane $P^*$ as the base locus. Every element of 
such a ruling is the Pl\"ucker embedding of the Grassmannian of the lines contained in a hyperplane through $P$. Notice also  that $Sing \ \sigma_P = P^*$. Then, since $S$ is a 
smooth complete intersection of three hyperplane sections of $G(1,4)$ and of a quadric section, it follows that $S \cap P^* = \emptyset$.  But then this ruling of quadrics of  $\sigma_P$ cuts
on $S$ a base pont free pencil $\vert D \vert$ such that $D^2 = 0$ and $DH = 4$. This is excluded again by a standard computation in the Picard lattice of a general Nikulin surface. 
\end{proof} \par
\section{\small Nikulin surfaces of genus 8 and symmetric cubic threefolds}
In what follows a symmetric cubic threefold is just a cubic hypersurface in $\mathbf P^4$  whose equation is the determinant of a symmetric $3 \times 3$ matrix of linear forms. As is well known
the family of symmetric cubic threefolds is irreducible and its quotient under the action of $PGL(5)$ has finitely many orbits. One of them is open and it is the projective equivalence class of
$\Sec B$, where $B$ is a rational normal quartic in $\mathbf P^4$. We will say that $\Sec B$ is \it the symmetric cubic threefold, \rm Moreover we fix $B$ and  the notation
$$
V := \Sec B.
$$
\par The symmetric cubic threefold is nicely related to special embeddings $A \subset G(1,4)$ of a rational normal sextic and to the family of Nikulin surfaces of genus $8$. To see this we go back
to the previous section, keeping the same notation. Since $S$ is a general Nikulin surface of genus $8$ we will assume, by 3.5 , that $\mathbb F_A$ is the Hirzebruch surface
$\mathbb F_0$ i.e. $\mathbb F_A = \mathbf P^1 \times \mathbf P^1$. We consider again the linear projection 
 $$
\alpha_A: R \to \mathbf P^4
$$
which is uniquely defined by $S$.  Fixing the sextic rational normal scroll $R \subset \mathbf P^7$, the family of the linear maps $\beta: R \to \mathbf P^4$ is parametrized by  the Grassmannian
$G(5,8)$. It follows from double points formula that:
\begin{lemma} Let $\beta: R \to \mathbf P^4$ be a general linear projection. Then $\beta$ is a generically injective morphism and $\Sing \beta(R)$ is a set of six non normal double points with two branches.
 \end{lemma}
Actually each point $o \in \Sing \beta(R)$ has embedding dimension $3$, since $R$ is a scroll, and quadratic tangent cone of rank two. It is easy to see that this is not the case for $\alpha_A: R \to \mathbf P^4$ and that the projected surface $ R' = \alpha_A(R) \subset \mathbf P^4$ has an interesting feature. Indeed we have
$$
A \subset S \subset T = G(1,4) \cap \mathbf P^6
$$
where the Nikulin surface $S$ is general. Then $S$ contains the \it eight \rm disjoint lines $N_1 \dots N_8$. Let $i = 1 \dots 8$, we observe that $N_i$ parametrizes a pencil of lines in $\mathbf P^4$. Moreover $AN_i = 2$ so that $N_i$ is
a bisecant line to $A$. \par On the other hand $R'$ is precisely the union of the lines parametrized by $A$. Let $o_i$ be the center of the pencil of lines parametrized by $N_i$ and let $Z_i = A \cdot N_i$.
Since $S$ is general  $Z_i$ is a $0$-dimensional scheme of length $2$ in $S$, moreover the next lemma is immediate
\begin{lemma} $\alpha_A/Z_i : Z_i \to \mathbf P^4$ contracts $Z_i$ to the point $o_i$. \end{lemma}
The lemma implies that $\alpha_A$ contracts the scheme $Z = \cup Z_i$,  of length $16$, to a scheme of length $8$ supported on the points $o_1 \dots o_8$. By the previous double points formula this is impossible if 
$\alpha_A$ is not an embedding on at most finitely many points. Hence $\Sing  R'$ contains a curve. \par  At first we observe that:
\begin{lemma} $\Sing R'$ is an integral curve of degree $m$ with $3 \leq m \leq 4$. \end{lemma}
\begin{proof} By the proof of theorem 3.5 $\alpha_A: R \to \mathbf P^4$ is a generically injective morphism, moreover $R'$ is not a cone nor contains curves of degree $ \leq 2$. On the other hand a general hyperplane section of $R'$ is an integral sextic curve in $\mathbf P^3$, hence the number of its singular points is $\leq 4$. \end{proof}
It is clear that $\alpha_A: R \to R'$ is the normalization map of $R'$. Let $\tilde B$ be the pull-back of $\Sing R'$ by $\alpha_E$, then $\tilde B$ is a curve of type $(a,b)$ in $R = \mathbf P^1 \times \mathbf P^1$ such that $a + 3b = 2m$. It is easy to compute the only possible cases:
  \begin{enumerate}
 \item $m = 4$ and $(a,b) = (2,2)$
 \item $m = 4$ and $(a,b) = (5,1)$
 \item $m = 3$ and $(a,b) = (3,1)$
 \end{enumerate}
  \par  Note that the type $(2,2)$ in case (1) implies that each line of $R'$ is bisecant to $\Sing R'$.  We do not discuss these cases. We only state as a \it claim \rm the next theorem,
  which describes what happens in the case of our interest.
  \begin{theorem} For a general Nikulin surface $S$ as above we have:
\begin{itemize}
\item[$\circ$] $\Sing R'$ is a rational normal quartic curve $B$ and $R'$ is the complete intersection of $V = \Sec B$ and a quadric through $B$.
 \end{itemize}
Moreover:
\begin{itemize}
\item[$\circ$] $R$ is $\mathbf P^1 \times \mathbf P^1$ and $\tilde B$ is a smooth curve of type $(2,2)$ in it,
\item[$\circ$] $T = < A > \cdot G(1,4)$ is a smooth quintic Del Pezzo threefold.
 \end{itemize} \end{theorem}
 In the next section an integral family $\mathbb P$ of Nikulin surfaces of genus $8$ is constructed, whose general member satisfies the conditions stated above. Then $\mathbb P$ is used to prove
 the rationality of the moduli space $\mathcal F^N_8$. \par A byproduct of the proof is  that $\mathbb P$ dominates $\mathcal F^N_8$, hence the previous theorem follows. The construction of
 a surface in this family offers an explicit geometric construction of a general Nikulin surface in genus $8$.
 
  \section{Special rational normal sextics in the $G(1,4)$}
To construct the required family and to use it later,  we construct at first a family of rational normal sextics $A \subset G(1,4)$ which are specially embedded. To begin we fix
a rational normal quartic curve $B \subset \mathbf P^4$ and the symmetric cubic threefold $V := \Sec V$. Let us also fix the notation
$$
\mathbf P^2 := \Hilb_2(B).
$$
Any point $z \in \mathbf P^2$ is an effective divisor $b_z \in \Div B$ of degree two, the line
$$
V_z := < b_z > 
$$
is a bisecant line to $B$. We will denote its parameter point in $G(1,4)$ by $z$. Let $\mathbb I \subset \mathbf P^4 \times G(1,4)$ be the universal line over $G(1,4)$, in it we have 
$$
\tilde V := \lbrace (x,z) \in \mathbf P^4 \times G(1,4) \ / \ x \in V_z \rbrace. 
$$
$\tilde V$ is endowed with its natural projection $\upsilon: \tilde V \to G(1,4)$, we set
$$
Z := \upsilon(\tilde V).
$$
Then we have $Z \subset G(1,4) \subset \mathbf P^9$, the next property is standard.
\begin{lemma}  \ \par
\begin{enumerate}
\item $\upsilon: \tilde V \to Z$ is the projective bundle $\mathbf PT_{\mathbf P^2}^*$.
\item $Z$ is the $3$-Veronese embedding of $\mathbf P^2$.
\item $Z$ is embedded in $G(1,4)$ as a surface of cohomology class $(3,6)$.
\end{enumerate}
\end{lemma}
Actually (2) and (3) follow from $c_1(T_{\mathbf P^2}) \cong \mathcal O_{\mathbf P^2}(3)$ and $c_2(T_{\mathbf P^2}) = 3$. \par On the other hand we consider the morphism
$\sigma: \tilde V \to V$, induced on $\tilde V$ by the projection in $\mathbf P^4$. The structure of $\sigma$ is well known as well.
\begin{lemma}  $\sigma: \tilde V \to V$ is the contraction to $B$  of the divisor
$$
E := \lbrace (x,z) \in V \times Z \ / \ x \in b_z \rbrace.
$$
Moreover $E$ in biregular to $B \times B$.
\end{lemma} 
\begin{proof} Let $x \in V - B$. Since $B$ is not degenerate of degree four, there exists a unique bisecant line $V_z$ passing  through $x$. Hence $\sigma$ is birational and $\sigma^{-1}(x) = (x,z)$.
Obviously $\sigma$ contracts $E$.  Finally let  $f: B \times B \to E$ be the morphism which is so defined: $f(x_1,x_2) = (x_1,z)$, where $b_z := x_1+x_2$. For every $z \in Z$ one has $B \cdot V_z = b_z$. This  implies that $f$ is biregular.  \end{proof}
Now we consider on $\tilde V$ the linear systems
$$
\vert H \vert := \vert \sigma^* \mathcal O_{\mathbf P^4}(1) \vert \ \text { and } \  \vert F \vert := \vert \upsilon^* \mathcal O_{\mathbf P^2}(1) \vert.
$$
We have $\Pic \tilde V \cong \mathbb Z[H] \oplus \mathbb Z[F]$ and we want to study with some detail 
$$
\vert 2H - E \vert.
$$
At first we observe that $(2H - E)F^2 = 0$. Indeed  every fibre $V_z \times \lbrace z \rbrace$ of $\upsilon$ has numerical class $F^2$ in the Chow ring $CH^*(\tilde V)$.
Since $B \cdot V_z = b_z$ we have $EF^2 = 2$ and hence $(2H - E)F^2 = 0$. On the other hand we have $(2H - E)H^2 = 2 H^3 = 6$. Hence it follows $2H - E \sim 2F$. Let $\mathcal I_B$
be the ideal sheaf of $B$ in $\mathbf P^4$, the next lemma then easily follows.
\begin{lemma} The linear system $\vert 2F \vert$ is the strict trasform of the linear system $\vert \mathcal I_B(2) \vert$ by the birational morphism $\sigma: \tilde V \to V$. \end{lemma}
It is clear from the lemma that to give $R \in \vert 2F \vert$ is equivalent to give  the pull-back by $\upsilon$ of a conic, embedded by the $3$-Veronese map of $\mathbf P^2$ as a sextic curve in $Z \subset G(1,4)$.
\begin{definition} $\mathbb A = \vert \mathcal O_{\mathbf P^2}(2) \vert $ is the family of embedded sextic curves 
$$
A \subset Z \subset G(1,4).
$$
We will say that $A \in \mathbb A$ is a special rational normal sextic of $G(1,4)$.
\end{definition}
In particular,  the linear space $< A >$ is a $\mathbf P^6$ for each $A \in \mathbb A$. For a given $A \in \mathbb A$ we will keep the following notations: $$ R = \upsilon^*A \ , \ R' = \sigma_*R \ , \ T := < A > \cdot G(1,4). $$
\par Now assume  $A$ is general.  We remark that the inclusion of $\mathbf P^1$-bundles $R \subset \tilde V$ is induced by the exact sequence
$$
0 \to \mathcal E_Z(-A) \to \mathcal E_Z \to \mathcal E_A \to 0,
$$
where $\mathcal E_Z = T_{\mathbf P^2}$ is the restriction of the universal bundle of $G(1,4)$ and hence $\mathcal E_A = T_{\mathbf P^2} \otimes \mathcal O_A$. Due to our generality assumption on
$A$, the splitting of $\mathcal E_A$ is balanced. This implies the next lemma.
\begin{lemma} $\mathcal E_A = \mathcal O_{\mathbf P^1}(3) \oplus \mathcal O_{\mathbf P^1}(3)$. \end{lemma}
 Next we want to study the morphism $ \upsilon / E: E \to \mathbf P^2$. Since $\mathbf P^2$ is $\Hilb_2(B)$ a conic in it is distinguished, namely the diagonal
$$
D  := \lbrace z \in \mathbf P^2 \ / \ b_z = 2x \rbrace.
$$
The double covering of $\mathbf P^2$ branched in $D$ can be viewed as follows.  Let
$$
P := \lbrace (p,t) \in \mathbf P^2 \times \mathbf P^{2*} \ / \ \text { $p \in t$ and $t$ is tangent to $D$} \rbrace.
$$
\par Then the projection $P \to \mathbf P^2$  is the 2:1 cover branched on $D$.  In particular its fibre at $p$ is the set  $\lbrace t', t'' \rbrace$ of the tangent lines
to $D$ passing through $p$.
  \begin{lemma} $\upsilon/ E: E \to \mathbf P^2$ is the double covering branched on $D$. \end{lemma}
 \begin{proof} Let $x \in B = \Sing V$ then $\sigma^*(x)$ is $\lbrace x \rbrace \times B$. Moreover $\upsilon(\sigma^*x)$ is the line
 $ t_x := \lbrace x+y, \ y \in B \rbrace \subset \mathbf P^2. $  As is well known $t_x$ is tangent to $D$ at $2x$. Let $z \in \mathbf P^2$ and let $b_z = x+y$, we just remark that
 the fibre of $\upsilon/E$ at $z$ is naturally bijective to $\lbrace t_x , t_y \rbrace$. This implies the statement.
\end{proof}
\begin{proposition} Let $A \in \mathbb A$ be general and let $\tilde B := R \cdot E$, then $R$ is $\mathbf P^1 \times \mathbf P^1$ and $\tilde B$ is a smooth curve in it of type $(2,2)$.
\end{proposition}
\begin{proof} By lemma 5.4 we have $R = \mathbf P(\mathcal E_A) = \mathbf P^1 \times \mathbf P^1$. Since $\vert R \vert$ is base point free, the rest of the statement is true for every $R$ transversal to $E$.
\end{proof} 
 For any $A \in \mathbb A$ we now consider the schematic image $R' := \sigma_*(R)$ of $R$. By lemma 5.3 $R'$ is a quadro cubic complete intersection. More precisely we have
$R' = V \cdot Q$,  where $Q$ is a quadric through $B$. The ruling of lines of $R'$ is parametrized by $A$.  
 In order to prove the next proposition the next lemma is useful.  The proof is an exercise on the geometry of $G(1,4)$ we omit.
\begin{lemma} A three dimensional linear section of $G(1,4)$ is not a smooth threefold iff it is contained in a codimension one Schubert cycle. \end{lemma}
  \begin{proposition} Let $A \in \mathbb A$ be general and let $T = < A > \cdot G(1,4)$. Then $T$  is a smooth quintic Del Pezzo threefold. \end{proposition}
\begin{proof} Since $A$ is general we can assume that $\mathcal E_A = \mathcal O_{\mathbf P^1}(3) \oplus \mathcal O_{\mathbf P^1}(3)$ and that $R' = V \cap Q$, with $Q$ a smooth quadric.
Assume that $T$ is singular, then $T$ is contained in a Schubert cycle of codimension one, that is, in the Chow variety of a plane $P \subset \mathbf P^4$. Now $V$ does not
contain any plane. To see this observe that $V$ is a hyperplane section of the secant variety $Sec \ Y$ of tte Veronese variety $Y$ in $\mathbf P^5$. Then  any hyperplane section through
a plane is the secant variety $V'$ of the degenerate quartic $B' \cup B''$, where $B', B''$ are smooth conics and $B' \cap B''$ is one point. On the other hand $Q$ does not contain any plane
as well, since it is smooth. Hence $P \cap Q \cap V$ is a curve of degree $m \leq 2$ contained in $R'$. It is easy to see that its pull back by $\sigma/R: R \to R'$ is a section of $R$ having self intersection $\geq -2$. This is impossible if $\mathcal E_A$ is balanced as above. Hence $T$ is smooth. \end{proof}
For a general $A \in \mathbb A$ we know that:  $R$ is $\mathbf P^1 \times \mathbf P^1$ and  $\sigma/R: R \to R'$ is the normalization map.  The curve $\tilde B$ $=$
$(\sigma/R)^*\Sing R'$ is smooth of type $(2,2)$.  $T$ is smooth. We can summarize the situation as follows.  
 \begin{theorem} For a general $A \in \mathbb A$ the statement of theorem  4.4 holds. \end{theorem}
 \section{Geometry of Nikulin surfaces of genus $8$}
 We have not yet used our family $\mathbb A$, of special embeddings of $\mathbf P^1$ in $G(1,4)$, to construct Nikulin surfaces of genus $8$, nor we have considered  the special feature of these
 embeddings. About this we can say in short  that a general $A \in \mathbb A$ admits a one dimensional family of bisecant lines which are contained in $T$. Moreover the union of them is a quadratic section
 of $T$. To see this quickly we fix a general $A \in \mathbb A$ and consider the complete intersection
 $$
 R' = Q \cap V.
 $$
 \par Since $A$ is general we can assume that $Q \in \vert \mathcal I_B(2) \vert $ is general, then it is known that $Q$ is a smooth quadric. We recall from \cite{FV, FV1} that \it the tangential quadratic complex of $Q$ \rm is just the family
$$
W \subset G(1,4)
$$
parametrizing the lines which are tangent to $Q$.  As is well known $W$ is a quadratic complex. In other words it is a quadratic section of $G(1,4)$. \par Actually $W$ is singular  and has two orbits under the action of $\PGL(5)$: $W - \Sing W$ and $\Sing \ W$, where the multiplicity is two. Finally: $$ \Sing W = F(Q), $$ where $F(Q)$ denotes the Hilbert scheme of lines in $Q$. $F(Q)$ is embedded in $G(1,4)$ as the image of $\mathbf P^3$ under the $2$-Veronese map.  
\begin{lemma}Ê$W$ does not contain $T$. \end{lemma}
\begin{proof}  Let $r, r'$ be disjoint lines of the ruling of $R' $ and $L := < r \cup r'>$. We consider the Grassmannian $G_L \subset G(1,4)$ of lines of $L$, which is embedded by its Pl\"ucker map. Since $r \cap r'$ is empty, it easily follows that $q_L := T \cap G_L$ is a conic containing the parameter points of $r$ and $r'$. Its corresponding quadric $Q'_L \subset L$ is not in $Q$. Otherwise we would have $q_L \subset \Sing W$, which is impossible because $< A > \cdot \Sing W = A$. This indeed follows because $A$ is the image of a  skew cubic of $\mathbf P^3$: since its ideal in $\mathbf P^3$  is generated by three independent quadrics, then $A$ is cut on $\Sing W$ by the codimension three linear space $< A >$. Since $r \cap r' = \emptyset$, it is also true that $Q'_L$ is either smooth or union of two disjoint planes. Let $Q_L = Q \cap L$, then $Q_L$ is  smooth and the intersection $Q_L \cdot Q'_L$ contains the skew lines $r$ and $r'$. Let $r'' \in \vert \mathcal O_{Q'_L}(r) \vert$ be general.  After the preceding remarks it is very easy to deduce that 
$r''$ is not tangent to $Q$. Hence $q_L$ is not in $W$ and $T$ is not in $W$.  \end{proof}
$W$ cuts on $T$ a surface which defines an element of the Hilbert scheme of K3 surfaces of genus $6$ in $G(1,4)$. We will see in a moment that this surface is a scroll in $T$ singular along $A$. To describe this scroll we consider the ruled surface over $B$ constructed as follows. \par For each $o \in B$ consider in $\mathbf P^2$ the divisor of degree two of $A$ $$ n_0 := \upsilon_*(R \cdot \sigma^*(o)). $$ Then $n_o$ defines the plane $P_o := < \sigma_*\upsilon^*n_o >$
and the pencil of lines of $P_o$ of center $o$. We denote this pencil by $N_o$. We also observe that $n_o$ is supported on two points for a general $o$ and that $\sigma_* \upsilon^* n_o$ is the union of the two lines in $R'$ passing through $o$. Finally we define the ruled surface
$$ 
S_A := \lbrace (o,n) \in A \times T \ / \ n \in N_o \rbrace
$$
and, via the projection $\tau: B \times T \to T$, its schematic image
$$
S'_A := \tau_*(S_A).
$$
\begin{definition}ÊWe say that $S'_A$ is the fake K3 surface of $A$. \end{definition}
A standard computation in the Chow ring of $G(1,4)$ shows that $S'_A$ has class $(4,6)$ in $CH^4(G(1,4))$. This implies that $S'_A$ is embedded in $T$ as a surface of degree ten.
On the other hand, since $T$ is a smooth quintic Del Pezzo threefold, $\Pic T$ is generated by $\mathcal O_T(1)$. Therefore we conclude that
$$
S'_A \in \vert \mathcal O_T(2) \vert.
$$
The next propositions say more.
\begin{proposition}Ê$S'_A$ is cut on $T$ by the tangential quadratic complex $W$. \end{proposition}
 \begin{proof} In view of the latter remarks, it suffices to show that $W$ contains $S'_A$. Let $(o,n) \in S_A$, then $n$ is a line of the pencil $N_o$. This is a pencil of  the plane
 $P_o$ and it is generated by two lines, say $n', n''$, of the scroll $R'$. Since $R' = Q \cdot V$ and $Q$ is smooth, it follows that $P_o \cdot Q = n' \cup n''$.
 Hence $P_o$ is a tangent plane to $Q$ at $o$ and each line of  the pencil $N_o$ of center $o$ is tangent to $Q$. This implies $N_o \subset S_A$ and hence $S'_A \subset W$.
 \end{proof}
 \begin{proposition}Ê$\Sing S'_A = A$ and $A$ has multiplicity two. \end{proposition}
\begin{proof} Let us consider again the morphism  $\tau: S_A \to S'_A$ and a point $n \in S'_A$. From the definition of $S_A$ it follows that  $\tau^*n$ is the scheme theoretic intersection of
$B$ and the line parametrized by $n$. This easily implies that $\tau$ is birational and, moreover, that $n$ is singular iff parametrizes a bisecant line to $B$ contained in $R'$, that is, $n \in A$
and $n$ has multiplicity two.
 \end{proof}
 \begin{remark} \rm We only mention, without proofs, how the scroll $S_A$ is related to the Fano variety $F(T)$ of lines of the smooth threefold $T$. It  is well known that every three dimensional linear section
 of $G(1,4)$ is projectively equivalent to $T$ and that $F(T)$ is a surface. The family of lines of $T$ has been studied in detail, cfr.\cite{FN, TZ}. Actually $F(T)$ is $\mathbf P^2$, let $u: \mathbb U \to \mathbf P^2$
 be the universal line and $\phi: \mathbb U \to G(1,4)$ the natural projection. For a plane curve $D \subset \mathbf P^2$ of degree $d$ we have that $\phi_* u^*D$ belongs to $\vert \mathcal O_T(d) \vert$. In
 particular we have $S_A = \phi_* u^* A'$, where $A'$ is a conic.
 \end{remark}
 
Finally we can pass to the construction of the predicted family of Nikulin surfaces of genus $8$. \par We start with a general $A \in \mathbb A$. Then $A$ is a rational normal sextic in the smooth threefold $T$.
Both $T$ and $A$ are projectively normal and  generated by quadrics.
Let $\mathcal I_A$ be the ideal sheaf of $A$ in $T$, we have $\dim \vert \mathcal I_A(2) \vert = 9$.
\begin{lemma}ÊA general $S \in \vert \mathcal I_A(2) \vert$ is a smooth K3 surface of genus $6$. \end{lemma}
\begin{proof}Ê Let $n \in A$ and $I_n = \lbrace S \in \vert \mathcal I_A(2) \vert \ / \ n \in \Sing S \rbrace$. Since $A$ is a smooth curve generated by quadrics, the codimension of  $I_n$ is two. Hence 
$\cup_{\small n \in A} I_n$ is a proper closed set and a general $S \in \vert \mathcal I_A(2) \vert$ is smooth along $A$. On the other hand a general $S$ is smooth on $T-A$ by Bertini theorem. \end{proof}
 We consider the morphism $\tau: S_A \to S'_A$ and the linear projection
$$
\tau^*: \vert \mathcal I_A(2) \vert \to \vert \tau^* \mathcal O_{S_A}(2) \vert.
$$
Since $S'_A \in \vert \mathcal I_A(2) \vert$ the image of $\tau^*$ is a linear system $\mathbb I$ of dimension $8$.   Since $A$ and $T$ are generated by quadrics, the fixed component of $\mathbb I$ is $\tau^*A$. 

\begin{lemma} Let $\pi: S_A \to B$ be the natural projection then 
$$ \mathbb I = \tau^*A +  \pi^* \vert \mathcal O_{\mathbf P^1}(8) \vert. $$
\end{lemma}
\begin{proof} Let $Q$ be a general quadric through $A$, not containing $S'_A$. Then $Q \cdot S'_A$ is a curve of degree 20 such that $A$ counts with multiplicity two. The residual component has
degree $8$. On the other hand let $o \in B$ and $N_o = \pi^*(o)$, then $N_o$ is embedded as a bisecant line to $o$. Indeed $N_o$ is the pencil of lines generated by the two lines of $R'$ passing
through $o$ and they define the two points of $N_o \cap A$. It is easy to deduce that the residual component is union of lines of the ruling of $S'_A$. Hence we have $\tau^*Q = \tau^*A + M$, where
$M \in \vert \mathcal O_{\mathbf P^1}(8) \vert$. The latter space has the same dimension of $\mathbb I$. Hence the statement follows. \end{proof}
Let $S \in \vert \mathcal I_A(2) \vert$ then $S$ is a smooth K3 surface. The lemma implies that $S$ contains eight disjoint lines $N_1, \dots, N_8$ so that $AN_i = 2$, $i = 1 \dots 8$, and
$$
S \cdot S'_A = 2A + N_1 + \dots + N_8.
$$
 We can definitely conclude, by theorem 2.7, that $S$ is a Nikulin surface.
\begin{theorem} Let $A \in \mathbb A$ be general, then a general $S \in \vert \mathcal I_A(2) \vert$ is a Nikulin surface of genus $8$. \end{theorem}
\section{The rationality of $\mathcal F^N_8$}
In this section we show the rationality of the moduli space $\mathcal F^N_8$ of genus 8 Nikulin surfaces and deduce theorem 4.4 from the method of proof.
\par To this purpose let us consider
in $\mathbb A \times G(1,4)$ the universal families $\mathcal R$ and $\mathcal T$, respectively parametrizing pairs $(A, z)$ such that $z \in A$ and $z \in T$,  $T :=$  $< A > \cdot G(1,4)$. With
some abuse we still denote by $\mathbb A$ a suitable open set of $\mathbb A$, parametrizing general and smooth rational normal sextics. Then we consider
the ideal sheaf $\mathcal I$ of $\mathcal R$ in $\mathcal T$ and the vector bundle 
$$
\mathcal V := \alpha_*( \mathcal I \otimes \beta^* \mathcal O_{G(1,4)}(2)).
$$
Here $\alpha$ and $\beta$ are the projections of $\mathbb A \times G(1,4)$ respectively onto $\mathbb A$ and $G(1,4)$. The construction defines the $\mathbf P^9$-bundle
$\alpha: \mathbb P \to \mathbb A$, where we set $$ \mathbb P := \mathbf P \mathcal V. $$
The fibre $\mathbb P_A$ of $\alpha$ at $A \in \mathbb A$ is precisely the linear system of Nikulin surfaces $\vert \mathcal I_A(2) \vert$ considered in theorem 6.6. On the other
hand we know that  $\mathbb A = \vert \mathcal O_{\mathbf P^2}(2) \vert $ and this is isomorphic to $\vert \mathcal I_B(2) \vert$ via the linear map
$$
\sigma_* \circ \upsilon^*: \mathbb A \to \vert \mathcal I_B(2) \vert.
$$
A general point of $\mathbb P$ is just a pair $(S, A)$ such that $S$ is a Nikulin surface of the linear system $\mathbb P_A$. More precisely its genus $8$ polarization is uniquely defined
from the pair as $\mathcal L = \mathcal H(A)$, where $\mathcal H = \mathcal O_S(1)$. Moreover we have $\mathcal H(-A) \cong \mathcal O_S(M)$ and  $2M \sim N_1 + \dots + N_8$, where 
the summands are the eight lines contained in $S \cap S_A$. We want to study the moduli map
$$
m: \mathbb P \to \mathcal F^N_8
$$
in order to prove that it is dominant. Then we consider two general elements $(S_1, A_1)$ and $(S_2, A_2)$ of $\mathbb P$. Let $ i = 1,2$: as above the pair $(S_i, A_i)$ uniquely defines the triple of
line bundles $(\mathcal L_i, \mathcal H_i, \mathcal A_i)$ where $\mathcal A_i = \mathcal O_{S_i}(A_i)$. Assume that the two elements have the same image by $m$. This is equivalent to say that there
exists an isomorphism  $f: S_1 \to S_2$ such that $f^*\mathcal H_2 \cong \mathcal H_1$ and $f^* \mathcal A_2 \cong \mathcal A_1$. Now recall that $\mathcal H_i$ uniquely defines the Mukai vector
bundle $\mathcal E_i$ of $S_i$ and that this is the restriction to $S_i$ of the universal bundle on $G(1,4)$. Hence it follows that $f^* \mathcal E_2 \cong \mathcal E_1$ and that $f$ defines an isomorphism
$$
f^*: H^0(\mathcal E_2) \to H^0(\mathcal E_1).
$$
This isomorphism and the property $f^* \mathcal A_1 \cong \mathcal A_1$ imply the next lemma.
\begin{lemma}Ê$(S_1, A_1)$  and $(S_2, A_2)$ have the same image by $m$ iff there exists $a \in \Aut G(1,4)$ such that $a(S_1) = S_2$ and $a(A_1) = A_2$. \end{lemma}
Due to the lemma we now study the group $G \subset \Aut G(1,4)$ leaving  the set of pairs $\mathbb P$ invariant. We have $\Aut G(1,4) = PGL(5)$: let $a \in G$ then $a$ is induced by an
automorphism of $\mathbf  P^4$. \par By definition $a$ leaves invariant the family of curves $\mathbb A$, therefore $G$ acts on $\mathbb A$ as well. In particular  $a$ leaves invariant the set of all  lines in $\mathbf P^4$ parametrized by $\cup A, \ A \in \mathbb A$. This is precisely the family of bisecant lines to the rational normal quartic $B$, hence $a$ leaves $B$ invariant. Since $a$ is the identity iff $a/B: B \to B$
is the identity, we can conclude as follows.
\begin{lemma}Ê$G = \Aut B = \PGL(2)$. \end{lemma}
Note that $\mathbb A$ is $\mathbf P^5$ and that $\dim \PGL(2) = 3$. Therefore the quotient of the action of $G$ on the base of the projection map $\alpha: \mathbb P \to \mathbb A$ is
rational. Indeed this quotient $\mathbb A / G$ is a surface. Therefore, since it is unirational, $\mathbb A / G$ is rational. \par It is useful to reconsider the action of $G$ on $\mathbb A$ as
follows. Recall that we have $\mathbb A = \vert \mathcal O_{\mathbf P^2}(2) \vert$, where $\mathbf P^2 = \Hilb_2(B)$. What is the action of $G$ on $\vert \mathcal O_{\mathbf P^2}(2) \vert$? To answer we recall that 
$\Hilb_2(B)$ contains a distinguished conic, namely the diagonal $D$  already considered in section 5. \par $D$
parametrizes the family of tangent lines to $B$ and this family is left invariant by $G$. Hence $a \in G$ acts on $\vert \mathcal O_{\mathbf P^2}(2) \vert$ as an  element $a' \in  \Aut \mathbf P^2$ leaving $D$ fixed. The map sending $a$ to $a'$ is an isomorphism of $G$ and $\Aut D$.  
\begin{lemma} The action of $G$ on $\mathbb A$ is generically faithful. \end{lemma}
\begin{proof} Let $G_A \subset G$ be the stabilizer of $A$ and $a \in G_A$. Consider $A$ as a general conic in $\mathbf P^2$. Then $a$ acts  as an automorphism of $\mathbf P^2$ such that $a(A) = A$ and $a(D) = D$. 
Hence it follows $a(A \cap D) = A \cap D$. But, since $A$ is general, the set of four points $A \cap D$ is general. Hence $a$ is the identity. \end{proof}
Let $\overline {\mathbb P}  := \mathbb P / G$ and $\overline{\mathbb A} := \mathbb A / G$, it follows from the previous lemmas that $\alpha: \mathbb P \to \mathbb A$ descends, over a non empty open set of $\overline {\mathbb A}$, to a $\mathbb P^9$-bundle
$$
\overline {\alpha}: \overline{\mathbb P} \to \overline {\mathbb A}.
$$
Since $\overline {\mathbb A}$ is rational, then $\overline {\mathbb P}$ is rational. Now the moduli map $m: \mathbb P \to \mathcal F^N_8$ descends to a rational map
$$
\overline m: \overline{\mathbb P} \to \mathcal F^N_8.
$$
Moreover lemma 7.1 implies that $\overline m$ is generically injective. Since $\overline {\mathbb P}$ and $\mathcal F^N_8$ are integral of the same dimension, it follows that $\overline m$ is birational.
This completes the proof of the next theorem.
\begin{theorem}ÊThe moduli space of genus 8 Nikulin surfaces is rational. \end{theorem}

\end{document}